\documentclass[12pt]{article}
\usepackage{geometry}                
\usepackage{graphicx}
\usepackage{amssymb}
\usepackage{epstopdf}
\usepackage{amsthm}
\usepackage{amsmath}
\usepackage[small]{complexity}
\usepackage{tikz}
\usepackage{anysize}
\usepackage{lscape}
\usepackage[latin1]{inputenc}
\usepackage{ytableau }
\newtheorem{thm}{Theorem}

\newtheorem{cor}{Corollary}
\newtheorem{lemma}{Lemma}
\newtheorem{definition}{Definition}

\title{Complexity of the Fourier transform \\ on the Johnson graph}

\author{Rodrigo Iglesias\\
\small Departamento de Matemática\\ [-0.8ex]
\small Universidad Nacional del Sur\\ [-0.8ex]
\small Bahía Blanca, Argentina.
 \and
Mauro Natale\\
\small Facultad de Ciencias Exactas\\ [-0.8ex]
\small Universidad Nacional del Centro\\ [-0.8ex]
\small Tandil, Argentina.} 

\begin{document}
\maketitle

\begin{abstract}
The set $X$ of $k$-subsets of an $n$-set has a natural graph structure where two $k$-subsets are connected if and only if the size of their intersection is $k-1$. This is known as the Johnson graph.
The symmetric group $S_n$ acts on the space of complex functions on $X$ and this space has a multiplicity-free decomposition as sum of irreducible representations of $S_n$, so it has a well-defined Gelfand-Tsetlin basis up to scalars. The Fourier transform on the Johnson graph is defined as the change of basis matrix from the delta function basis to the Gelfand-Tsetlin basis.

The direct application of this matrix to a generic vector requires $\binom{n}{k}^2$ arithmetic operations. 
We show that --in analogy with the classical Fast Fourier Transform on the discrete circle-- this matrix can be factorized as a product of $n-1$ orthogonal matrices, each one with at most two nonzero elements in each column. 
This factorization shows that the number of arithmetic operations required to apply this matrix to a generic vector is bounded above by $2(n-1) \binom{n}{k}$. 
As a consequence, we show that the problem of computing all the weights of the irreducible components of a given function can be solved in $O(n \binom{n}{k})$ operations, improving the previous bound $O(k^2 \binom{n}{k})$ when $k$ asymptotically dominates $\sqrt{n}$ in a non-uniform model of computation. The same improvement is achieved for the problem of computing the isotypic projection onto a single component.

The proof  is based on the construction of $n-1$ intermediate bases, each one parametrized by certain pairs composed by a  standard Young tableau and a word. The parametrization of each basis is obtained via the Robinson-Schensted insertion algorithm.


\end{abstract}

\section{Introduction}
The set of all subsets of cardinality $k$ of a set of cardinality $n$ is a basic combinatorial object with a natural metric space structure where two $k$-subsets are at distance $d$ if the size of their intersection is $k-d$. This structure is captured by the Johnson graph $J(n,k)$, whose nodes are the $k$-subsets and two $k$-subsets are connected if and only if they are at distance $1$. 

The Johnson graph is closely related to the Johnson scheme, an association scheme of major significance in classical coding theory (see \cite{Delsarte} for a survey on association scheme theory and its application to coding theory). 
Recently, the Johnson graph played a fundamental role in the breakthrough quasipolynomial time algorithm for the  graph isomorphism problem presented in \cite{Babai} (see \cite{Toran} for background on the graph isomorphism problem).

  Functions on the Johnson graph arise in the analysis of ranked data. In many contexts, agents choose a $k$-subset from an $n$-set, and the data is collected as the function that assigns to the $k$-subset  $x$ the number of agents who choose $x$. This situation is considered, for example, in the statistical analysis of certain lotteries (see \cite{Diaconis2}, \cite{Diaconis-Rockmore}).
  
The vector space of functions on the Johnson graph is a representation of the symmetric group and it  decomposes as a multiplicity-free direct sum of irreducible representations (see \cite{Stanton}). 
Statistically relevant information about the function is contained in the isotypic projections of the function onto each irreducible component. This approach to the analysis of ranked data was called spectral analysis by Diaconis and developed in \cite{Diaconis}, \cite{Diaconis2}.
The problem of the efficient computation of the isotypic projections  has been studied by Diaconis and Rockmore in \cite{Diaconis-Rockmore}, and by Maslen, Orrison and Rockmore in \cite{Lanczos}. 
  
  The classical Discrete Fourier Transform (DFT) on the cyclic group  $\mathbb{Z}/2^n \mathbb{Z}$ can be seen as the application of a change of basis matrix from the  basis  $B_0$ of delta functions  to the basis $B_n$ of characters of the group $\mathbb{Z}/2^n \mathbb{Z}$. The direct application of this matrix to a generic vector involves $(2^n)^2$ arithmetic operations. The Fast Fourier Transform (FFT) is a fundamental algorithm that computes the DFT in $O(n 2^n )$ operations. This  algorithm was discovered by Cooley and Tukey \cite{Cooley-Tukey} and  the efficiency of their algorithm is due to a factorization of the change of basis matrix
  $$
  [B_0]_{B_n} = [B_{n-1}]_{B_n}\  ... \ [B_1]_{B_2} \ [B_0]_{B_1}
  $$
  where $B_1 ,..., B_{n-1}$ are intermediate orthonormal bases such that each matrix $[B_{i-1}]_{B_i}$ has  two nonzero entries in each column.
  
  In this paper, we show that the same phenomenon occurs in the case of the non-abelian Fourier transform on the Johnson graph.  This transform is defined as the application  of the change of basis matrix from the  basis  $B_0$ of delta functions  to the basis $B_n$ of Gelfand-Tsetlin functions. The Gelfand-Tsetlin basis --defined in Section \ref{GT basis}--  is well-behaved with respect to the action of the symmetric group $S_n$, in the  sense that each irreducible component is generated by a subset of the basis. 
    
  A direct computation of this Fourier transform involves $\binom{n}{k}^2$ arithmetic operations.
  The computational model used here counts a single complex multiplication
and addition as one operation.
  We construct intermediate orthonormal bases $B_1 ,..., B_{n-1}$ such that each change of basis matrix $[B_{i-1}]_{B_i}$ has at most two nonzero entries in each column. 
  Each intermediate basis $B_i$ is parametrized by pairs composed by a standard Young tableau of height at most two and a word in the alphabet $\{ 1,2 \}$ as shown in Figure \ref{Labels of intermediate bases}.
  These intermediate bases enable the computation of the non-abelian Fourier transform --as well as its inverse-- in at most $2(n-1) \binom{n}{k}$ operations.

\begin{figure}[ht] 
\label{Labels of intermediate bases}
\begin{center}
\includegraphics[scale=0.6]{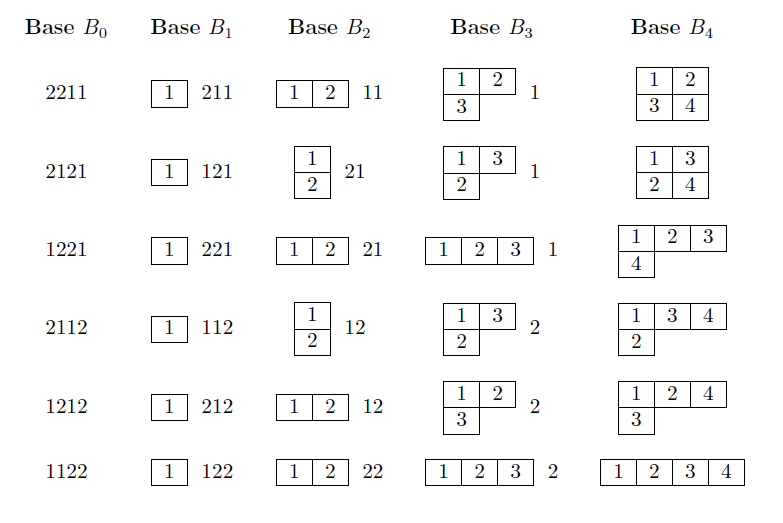} 
\caption{Labels of the intermediate bases in the case $n=4, k=2$. The $i$-th column parametrize the basis $B_i$. }
\label{Labels of intermediate bases}
\end{center}
\end{figure}

 The upper bound we obtained for the algebraic complexity of the Fourier transform on the Johnson graph can be applied to the well-studied problem of computing the isotypic components of a function.
The most efficient algorithm for computing all the isotypic components --given by Maslen, Orrison and Rockmore in \cite{Lanczos}--  relies on Lanczos iteration method and uses $O(k^2 \binom{n}{k})$ operations. 

But if the problem were to compute the isotypic projection onto a single component, it is no clear how to reduce this upper bound using the algorithm in \cite{Lanczos}. We show that -once the intermediate matrices $[B_{i-1}]_{B_i}$ have been computed for a fixed pair $(n,k)$-- this task can be accomplished in $O(n \binom{n}{k})$ operations, so our upper bound is an improvement when $k$ asymptotically dominates $\sqrt{n}$. 

We point out that this is an improvement only within the non-uniform setting, in the sense that our results imply the existence of smaller arithmetic circuits for each pair $(n,k)$, but we do not present an algorithm to find  each circuit because we do not give an algorithm to find the intermediate matrices $[B_{i-1}]_{B_i}$. We refer to \cite{Papadimitriou} (Section 11.4), \cite{Arora} (Chapter 14) and \cite{Burgisser} for an introduction to non-uniform models of computation and arithmetic circuits.

We also show that the same $O(n \binom{n}{k})$ bound is achieved for the problem of computing all the weights of the irreducible components appearing in the decomposition of a function.  This problem could also be solved by computing every isotypic component and measuring their lengths, but this approach requires $O(k^2 \binom{n}{k})$ operations if we use the algorithm in \cite{Lanczos}.

  In Section \ref{GT basis}, we review the definition  of Gelfand-Tsetlin bases for representations of the symmetric group. In Section  \ref{Johnson graph}, we describe the well-known decomposition of the function space on the Johnson graph and define the corresponding Gelfand-Tsetlin basis. In Section \ref{Intermediate}, we introduce the sequence of intermediate bases of the function space, we analyze the sparsity of the change of basis matrix between two consecutive bases and then we prove the main result, contained in Theorem \ref{Main theorem}.  In Section \ref{RS} we point out the relation of our algorithm with the Robinson-Schensted insertion algorithm. In Section \ref{Applications}  we apply our algorithm to the problem of the computation of the isotypic components of a function on the Johnson graph. Finally, in Section \ref{Further directions} we propose a problem for future work.

\section{Gelfand-Tsetlin bases}
\label{GT basis}

Consider the chain of subgroups of $S_n$ 
$$
S_1 \subset S_2 \subset S_3 ...\subset S_n
$$
where $S_k$ is the subgroup of those permutations fixing the last $n-k$ elements of $\{1,...,n\}$.
Let $Irr(n)$ be the set of equivalence classes of irreducible complex representations of $S_n$. 
A fundamental fact in the representation theory of $S_n$ is that if $V_{\lambda}$ is an irreducible $S_n$-module corresponding to the representation $\lambda \in Irr(n)$ and we consider it by restriction as an $S_{n-1}$-module, then it decomposes as sum of irreducible representations of $S_{n-1}$ in a multiplicity-free way (see for example \cite{Vershik}). This means that if $V_{\mu}$ is an irreducible $S_{n-1}$-module corresponding to the representation $\mu \in Irr(n-1)$ then the dimension of the space 
$
\mbox{Hom}_{S_{n-1}} (V_{\mu},V_{\lambda})
$
is $0$ or $1$.
The \textit{branching graph}  is the following directed graph. The set of nodes is the disjoint union 
$$
\bigsqcup_{n \ge 1}Irr(n).
$$
Given representations $\lambda \in Irr(n)$ and $\mu \in Irr(n-1)$ there is an edge connecting them if and only if  
$\mu$ appears in the decomposition of $\lambda$, that is, if $\mbox{dim Hom}_{S_{n-1}} (V_{\mu},V_{\lambda})=1$.
If there is an edge between them we write
$$
\mu \nearrow \lambda,
$$
so we have a canonical decomposition of $V_{\lambda}$ into irreducible $S_{n-1}$-modules
$$
V_{\lambda} = \bigoplus_{\mu \nearrow \lambda} V_{\mu} .
$$
Applying this formula iteratively we obtain a uniquely determined decomposition into one-dimensional subspaces 
$$
V_{\lambda} = \bigoplus_{T} V_{T} ,
$$
where $T$ runs over all chains
$$
T = \lambda_1 \nearrow \lambda_2 \nearrow ... \nearrow \lambda_n,
$$
with $\lambda_i \in Irr(i)$ and $\lambda_n = \lambda$.
Choosing a unit vector $v_T$ --with respect to the $S_n$-invariant inner product  in $V_\lambda$--  of the one-dimensional space $V_T$ we obtain a basis $\{v_T\}$ of the irreducible module $V_{\lambda}$, which is called the \textit{Gelfand-Tsetlin basis}. 

Observe that if $V$ is a multiplicity-free representation of $S_n$ then there is a uniquely determined --up to scalars--  Gelfand-Tsetlin basis of $V$. In effect, if 
$$
V = \bigoplus_{\lambda \in S \subseteq Irr(n)} V_{\lambda}
$$
and  $B_{\lambda}$ is a GT-basis of $V_{\lambda}$ then a GT-basis of $V$ is given by the disjoint union 
$$
B = \bigsqcup_{\lambda \in S \subseteq Irr(n)} B_{\lambda}.
$$

The Young graph is the directed graph where the nodes are the Young diagrams and there is an arrow from  $\lambda$ to  $\mu$ if and only if $\lambda$ is contained in $\mu$ and their difference consists in only one box. It turns out that the branching graph is isomorphic to the Young graph and there is a bijection between the set of Young diagrams with $n$ boxes and $Irr(n)$ (Theorem 5.8 of \cite{Vershik}). 
  
Then there is a bijection between the Gelfand-Tsetlin basis of $V_{\lambda}$ --where $\lambda$ is a Young diagram-- and the set of paths in the Young graph starting at  the one-box diagram and ending at the diagram $\lambda$. Each path can be represented by a unique standard Young tableau, so that the Gelfand-Tsetlin-basis of $V_{\lambda}$ is parametrized by the set of standard Young tableaux of shape $\lambda$ (see Figure \ref{Young graph}).
From now on we identify a chain 
$\lambda_1 \nearrow \lambda_2 \nearrow ... \nearrow \lambda_n$
with its corresponding standard Young tableau.

\begin{figure}[ht] 
\begin{center}
\includegraphics[scale=0.4]{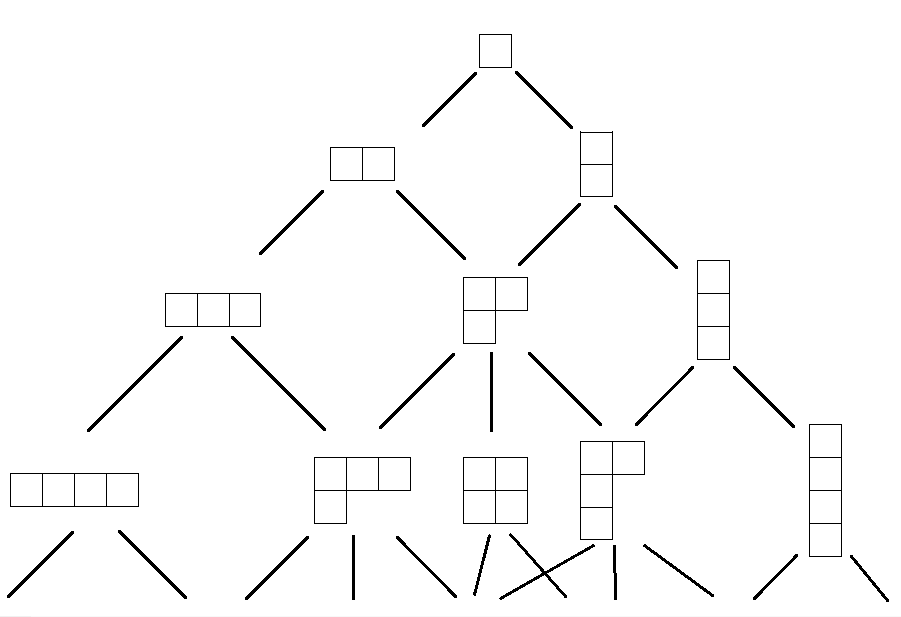} 
\caption{The Young graph. Each path from the top node to a particular Young diagram $\lambda$ is identified with a standard Young tableau of shape $\lambda$.
}
\label{Young graph}
\end{center}
\end{figure}

\section{Decomposition of the function space on the Johnson graph}
\label{Johnson graph}

 We define a  \textit{$k$-set} as a subset of $\{1,...,n\}$ of cardinality $k$. Let   $X$ be the set of all $k$-sets.  Given two $k$-sets $x,y$ the distance  $d(x,y)$  is defined as $n- \lvert x \cap y \rvert$. The group $S_n$ acts naturally on  $X$ by
$$
\sigma \{i_1,...,i_k\} = \{\sigma(i_1),...,\sigma(i_k)\}
$$
The vector space $\mathcal{F}$ of the complex valued functions on $X$ is a complex representation of  $S_n$ where the action is given by $\sigma f = f \circ \sigma^{-1}$.

 To each $k$-set $x \in X$ we attach the delta function $\delta(x)$ defined on $X$ by 
$$
\delta(x) (z) = 
\left\{ 
\begin{array}{cc}
1& \mbox{if} \ \  x=z \\
0 & \mbox{otherwise}
\end{array}
\right. .
$$

We consider $\mathcal{F}$ as an inner product space where the inner product is such that the delta functions form an orthonormal basis.

A Young diagram can be identified with the sequence given by the numbers of boxes in the rows, written top down. For example the Young diagram
\ytableausetup{smalltableaux}
\begin{center}
\ydiagram{5,4,2}
\end{center}
is identified with $(5,4,2)$.
It can be shown (see \cite{Stanton}) that the decomposition of $\mathcal{F}$ as a direct sum of irreducible representations of $S_n$ is given as follows.

From now on we denote by $s$ the number $min(k,n-k)$.

\begin{thm}
\label{decomposition of F}
The space  $\mathcal{F}$ of functions on the Johnson graph J(n,k) decomposes in  $s+1$ multiplicity-free irreducible representations of the group $S_n$. Moreover, the decomposition is given by
$$
\mathcal{F} = \bigoplus_{i=0}^{s} V_{\alpha_i}
$$ 
where $\alpha_i$ is the Young diagram $(n-i,i)$.
\end{thm}

For example, if $n=6$ and $k=2$ then  the irreducible components of $\mathcal{F}$ are in correspondence with the Young diagrams
\ytableausetup{smalltableaux}
\begin{center}
 \ \ 
\ydiagram{6,0} \  \ \ \ 
\ydiagram{5,1} \  \ \ \ 
\ydiagram{4,2} 
\end{center}

\subsection{Gelfand-Tsetlin basis of $\mathcal{F}$}
\label{Gelfand-Tsetlin basis of F}

From Theorem \ref{decomposition of F} we see that $\mathcal{F}$ has a well-defined --up to scalars-- Gelfand-Tsetlin basis and that there is a bijection between the set of elements of this GT-basis and the set of standard tableaux of shape $(n-a,a)$ where $a$ runs from $0$ to $s$.

Let us give a more explicit description of the GT-basis of $\mathcal{F}$. 
Consider the space $ \mathcal{F}$ as an $S_i$-module for $i=1,...,n$, and 
 let 
 $\mathcal{F}_{i,\lambda} $ be the isotypic component corresponding to the irreducible representation $\lambda$ of $S_i$ so that for each $i$ we have a decomposition
 $$
 \mathcal{F}= \bigoplus_{\lambda \in Irr(S_i)} \mathcal{F}_{i,\lambda}
 $$
 where $\mathcal{F}_{i,\lambda} \perp \mathcal{F}_{i,\lambda'}$ if $\lambda \neq \lambda'$.
For each standard tableau $\lambda_1 \nearrow \lambda_2 \nearrow ... \nearrow \lambda_n$ let
$$
 \mathcal{F}_{\lambda_1 \nearrow \lambda_2 \nearrow ... \nearrow \lambda_n} 
 =
  \mathcal{F}_{1,\lambda_1} \cap \mathcal{F}_{2,\lambda_2} \cap ... \cap
 \mathcal{F}_{n,\lambda_n}
$$
Then Theorem \ref{decomposition of F} shows that $\mathcal{F}$ has an orthogonal decomposition in one-dimensional subspaces
$$
 \mathcal{F} =  \bigoplus_{\lambda_1 \nearrow \lambda_2 \nearrow ... \nearrow \lambda_n}
 \mathcal{F}_{\lambda_1 \nearrow \lambda_2 \nearrow ... \nearrow \lambda_n}
 $$
where $\lambda_n$ runs through all representations of $S_n$ corresponding to Young diagrams $(n-a,a)$ for $a=0,...,s$ (see Figure \ref{GT of the Johnson graph}).

\begin{figure}[ht] 
\begin{center}
\includegraphics[scale=0.45]{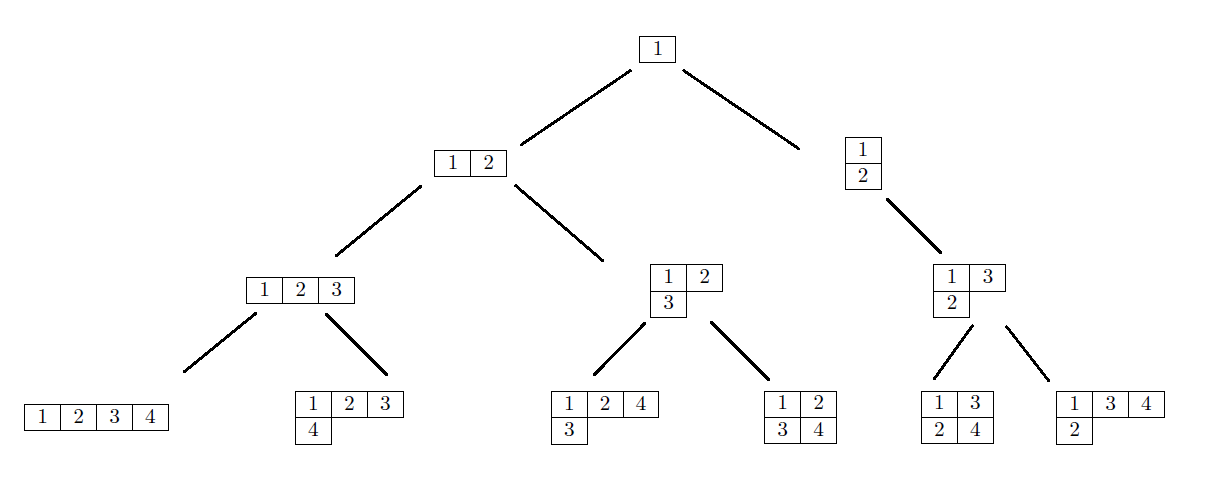} 
\caption{The leaves of this tree parametrize the Gelfand-Tsetlin basis of the space $\mathcal{F}$ of functions on the Johnson graph $J(4,2)$.
}
\label{GT of the Johnson graph}
\end{center}
\end{figure}

\section{The intermediate bases}
\label{Intermediate}

Let us describe schematically the Fast Fourier Transform algorithm for the Johnson graph. The input is a vector $f$ in the space $\mathcal{F}$ of functions on the set $X$ of $k$-sets, written in the delta function basis $B_0$, given as a column vector $[f]_{B_0}$. The output of the algorithm is a column vector representing the  vector $f$ written in the basis $B_n$, the Gelfand-Tsetlin basis of $\mathcal{F}$. In other words, the objective is to apply  the change of basis matrix to a given column vector:
$$
[f]_{B_n} = [B_0]_{B_n} \  [f]_{B_0}
$$
Our technique to realize this matrix multiplication is to construct a sequence of intermediate orthonormal bases $B_1, B_2, ..., B_{n-1}$ such that 
$$
[B_0]_{B_n} = [B_0]_{B_1}\ \  [B_1]_{B_2} \ \ ...\ \ [B_{n-1}]_{B_n} 
$$
is a decomposition where each factor is a very sparse matrix.

 \subsection{Definition of the basis $B_i$}
 
 We represent a $k$-set by a word in the alphabet $\{1,2\}$ as follows. The element $i \in \{1,...,n\}$ belongs to the $k$-subset if and only if the place $i$ of the  word is occupied by the letter $1$. For example, 
$$
\{2,3,6,8\} \subseteq \{1,...,9\} \ \ \ \ \ \rightarrow \ \ \ \ \ \  2\ 1\ 1\ 2\ 2\ 1\ 2\ 1\ 2
$$
So, from now on, we identify $X$ with the set of words of length $n$ in the alphabet $\{1,2\}$ such that the letter $1$ appears $k$ times.  The group $S_n$ acts on $X$ in the natural way.

 
 The subgroup $S_i$ with $1 \leq i \leq n$ acts on the first $i$ letters fixing the last $n-i$ letters of the word. 
Let $X_i$ be the set of words --in the alphabet $\{1,2\}$-- of length $n-i$ where the letter $1$ appears $k$ times or less.
 
For each  $w \in X_i$  we define $X^w$ as the subset of those words $x \in X$  that consist of a concatenation $x=w'w$ for some word $w'$ of length $i$. Observe  that each subset $X^w$ is stabilized by the action of the subgroup $S_i$.

 Let $\mathcal{F}^w$ be the subspace of $\mathcal{F}$ spanned by the delta functions $\delta(x)$ such that $x \in X^w$:
 $$
 \mathcal{F}^w = \bigoplus_{x \in X^w} \mathbb{C} \ \delta(x).
 $$
 Then $\mathcal{F}$ decomposes as
 $$
 \mathcal{F} = \bigoplus_{w \in X_i} \mathcal{F}^w
 $$
 and each subspace $\mathcal{F}^w$ is invariant by the action of $S_i$.
 
 The following is a key observation. Suppose that the letter $1$ appears $k-r$ times in the word $w$, where $0 \leq r \leq k$. Then the subset $X^w$ consists of those words of the form $w' w$ such that  $w'$ is of  length $i$ and it has exactly $r$ appearances of the letter $1$. This means that $X^w$ has the structure of the Johnson graph $J(i,r)$ and, when acted by the subgroup $S_i$, the space of $\mathbb{C}$-valued functions on $X^w$ decomposes as an $S_i$-module in a multiplicity-free way according to the  formula of Theorem \ref{decomposition of F}.
 
 As a consequence, each  subspace $\mathcal{F}^w$ has a Gelfand-Tsetlin decomposition
 $$
 \mathcal{F}^w =  \bigoplus_{\lambda_1 \nearrow \lambda_2 \nearrow ... \nearrow \lambda_i}
 \mathcal{F}^w_{\lambda_1 \nearrow \lambda_2 \nearrow ... \nearrow \lambda_i},
 $$
 where $\lambda_i$ runs over all Young diagrams $(i-a , a)$ with $0 \leq a \leq min(r,i-r)$.
Then $\mathcal{F}^w$ has a Gelfand-Tsetlin basis $B^w$, uniquely determined up to scalars.
We define, up to scalars,  the $i$-th intermediate basis of $\mathcal{F}$ as
 $$
 B_i = \bigsqcup_{w \in X_i} B^w.
 $$
 
From  Theorem \ref{decomposition of F}, we see that the basis $B^w$ is parametrized by the set of  standard tableaux of shape $(i-a , a)$ with $0 \leq a \leq min(r,i-r)$. On the other hand, the word $w$ runs over the set $X_i$. Figure \ref{Intermediate bases} illustrates the structure of the intermediate bases.

 \begin{figure}[ht] 
\begin{center}
\includegraphics[scale=0.6]{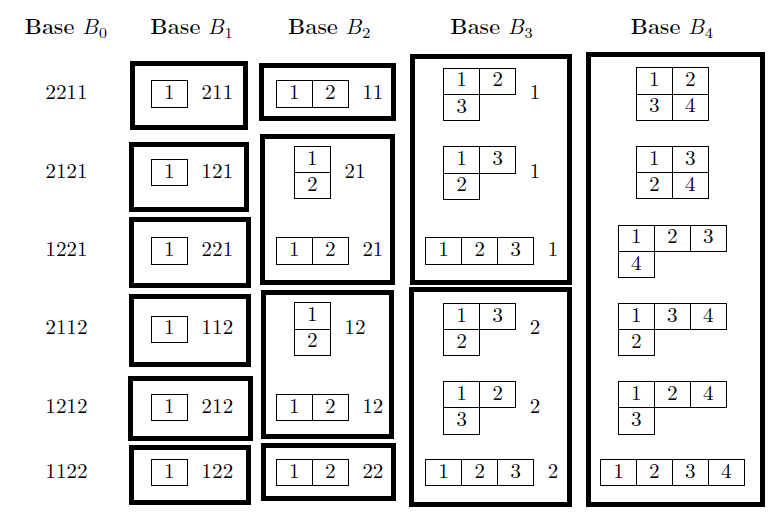} 
\caption{Labels of the intermediate bases in the case $n=4, k=2$.  The boxes in each column represent the decomposition $B_i = \sqcup_{w \in X_i} B^w$.
}
\label{Intermediate bases}
\end{center}
\end{figure}

 \subsection{Sparsity of the change of basis matrix $[B_i]_{B_{i-1}}$}
 Let $\mathcal{F}_{i,\lambda} $ be defined as in Section \ref{Gelfand-Tsetlin basis of F}.
  On the other hand, for $i=1,...,n$ and $c \in \{1,2\}$ let us define $\mathcal{F}^{i,c}$ as the subspace of $\mathcal{F}$ generated by the delta functions $\delta(x)$ such that the word $x$ has the letter $c$ in the place $i$. For each $i$ we have a decomposition 
 $$
 \mathcal{F}=  \mathcal{F}^{i,1} \oplus \mathcal{F}^{i,2}
 $$
 with $ \mathcal{F}^{i,1} \perp \mathcal{F}^{i,2}$.
 Let $w \in X_i$  
 $$
 w= c_{i+1} c_{i+2} ... c_n.
 $$
 Then 
 $$
  \mathcal{F}^w = \mathcal{F}^{i+1,c_{i+1}}  \cap ... \cap
 \mathcal{F}^{n-1,c_{n-1}} \cap \mathcal{F}^{n,c_{n}}
$$
 and the one-dimensional subspaces generated by elements of the base $B_i$ can be expressed as
 $$
 \mathcal{F}^w_{\lambda_1 \nearrow \lambda_2 \nearrow ... \nearrow \lambda_i}
=
 \mathcal{F}_{1,\lambda_1} \cap \mathcal{F}_{2,\lambda_2} \cap ... \cap
 \mathcal{F}_{i,\lambda_i} \cap \mathcal{F}^{i+1,c_{i+1}}  \cap ... \cap
 \mathcal{F}^{n-1,c_{n-1}} \cap \mathcal{F}^{n,c_{n}}.
$$

Let us compare the base $B_i$ with the base $B_{i-1}$. 

\begin{definition}
Let $b \in B_{i-1}$ and $b' \in B_{i}$. We say that $b$ and $b'$ are \textit{related} if both belong to the subspace 
$$
\mathcal{F}_{1,\lambda_1}  \cap ... \cap  \mathcal{F}_{i-1,\lambda_{i-1}}  \cap \mathcal{F}^{i+1,c_{i+1}}  \cap ... \cap
 \mathcal{F}^{n,c_{n}}
$$
for some Young diagrams $ \lambda_1,..., \lambda_{i-1}$ and some letters $c_{i+1},..., c_n$ in $\{1,2\}$.
\end{definition}

\begin{lemma}
\label{orthogonality}
Suppose that the subspace 
$$
\mathcal{F}_{1,\lambda_1}  \cap ... \cap  \mathcal{F}_{i-1,\lambda_{i-1}} \cap
 \mathcal{F}_{i,\lambda_i} \cap \mathcal{F}^{i+1,c_{i+1}}  \cap ... \cap
 \mathcal{F}^{n,c_{n}}
 $$
 generated by an element of the basis $B_i$ is not orthogonal to the subspace
 $$
\mathcal{F}_{1,\lambda'_1}  \cap ... \cap  \mathcal{F}_{i-1,\lambda'_{i-1}} \cap
 \mathcal{F}^{i,c'_i} \cap \mathcal{F}^{i+1,c'_{i+1}}  \cap ... \cap
 \mathcal{F}^{n,c'_{n}}
 $$
 generated by an element of the basis $B_{i-1}$.
 Then 
 $$
 \lambda_1 = \lambda'_1, \ \ ...\ \ , \lambda_{i-1} = \lambda'_{i-1}
 \ \ \ \ \ \  and \ \ \ \ \ \ \ 
 c_{i+1}=c'_{i+1}, \ \ ... \ \ , c_n=c'_n.
 $$
\end{lemma}
\begin{proof}
If $ \lambda_j \neq \lambda'_j$ for some $j$ with $1\leq j \leq i-1$, then  $$\mathcal{F}_{j,\lambda_j} \perp \mathcal{F}_{j,\lambda'_j}$$ and the two subspaces are orthogonal. 
If $ c_j \neq c'_j$ for some $j$ with $i+1\leq j \leq n$, then  $$\mathcal{F}^{j,c_j} \perp \mathcal{F}^{j,c'_j}$$ and the two subspaces are orthogonal.
\end{proof}

It is immediate that Lemma \ref{orthogonality} is equivalent to the following.
\begin{cor}
If $b \in B_{i-1}$ and $b' \in B_{i}$ are not orthogonal then they are related.
\end{cor}

 \begin{cor}
 \label{sparsity}
 For each element $b$ of the base $B_i$ there is at most two elements of the base $B_{i-1}$ not orthogonal to $b$. 
 \end{cor}
\begin{proof}
Suppose that the element $b$ of $B_i$ generates the subspace 
$$
\mathcal{F}_{1,\lambda_1}  \cap ... \cap  \mathcal{F}_{i-1,\lambda_{i-1}} \cap
 \mathcal{F}_{i,\lambda_i} \cap \mathcal{F}^{i+1,c_{i+1}}  \cap ... \cap
 \mathcal{F}^{n,c_{n}}.
$$
By Lemma \ref{orthogonality}, the subspace generated by an element of $B_{i-1}$ not orthogonal to $b$ is of the form 
$$
\mathcal{F}_{1,\lambda_1}  \cap ... \cap  \mathcal{F}_{i-1,\lambda_{i-1}} \cap
 \mathcal{F}^{i,c_i} \cap \mathcal{F}^{i+1,c_{i+1}}  \cap ... \cap
 \mathcal{F}^{n,c_{n}},
 $$
 where $c_i =1 $ or  $c_i =2 $.
\end{proof}

\begin{thm}
\label{Main theorem}
Let $B_0$ be the delta function basis of $\mathcal{F}$ and let $B_n$ be a Gelfand-Tsetlin basis of $\mathcal{F}$. We assume that the matrices $[B_{i-1}]_{B_{i}}$ for $i=2,3,...,n$ have been computed. Then, given a column vector $[f]_{B_0}$ with $f \in \mathcal{F}$, the column vector $[f]_{B_n}$ given by 
$$
[f]_{B_n} = [B_0]_{B_n} \  [f]_{B_0}
$$
can be computed using at most $2(n-1) \binom{n}{k}$ operations.
\end{thm}
\begin{proof}
By Corollary \ref{sparsity} we see that each column of the matrix 
$[B_{i-1}]_{B_{i}}$ has at most two non-zero elements. Then the application of  the matrix $[B_{i-1}]_{B_{i}}$ to a generic column vector can be done using at most $2 \binom{n}{k}$ operations. 
Observe that $[B_0]_{B_1}$ is the identity matrix. We have 
\begin{eqnarray*}
[B_0]_{B_n} &= [B_0]_{B_1}\ \  [B_1]_{B_2} \ \ ...\ \ [B_{n-1}]_{B_n}  = [B_1]_{B_2} \ \ ...\ \ [B_{n-1}]_{B_n}.
\end{eqnarray*}
Then the successive applications of the $n-1$ matrices can be done in at most $2(n-1) \binom{n}{k}$ operations.
\end{proof}

\subsection{Example}
Consider the case $n=4$, $k=2$. For each vector $b$ of the basis $B_{i}$, there exists a unique word  $w \in X_i$ and a unique standard tableau $\lambda_1 \nearrow \lambda_2 \nearrow ... \nearrow \lambda_i$ such that $b \in \mathcal{F}^w_{\lambda_1 \nearrow \lambda_2 \nearrow ... \nearrow \lambda_i}$. Then $b$ is a linear combination of those elements of $B_{i-1}$ that belong to the space $\mathcal{F}^{\bar{w}}_{\lambda_1 \nearrow \lambda_2 \nearrow ... \nearrow \lambda_{i-1}}$, where $\bar{w}=1w$ or $\bar{w}=2w$. Then the matrices $[B_i]_{B_{i-1}}$  have the form

\[
\left[  B_{1}\right]  _{B_{2}}=\left[
\begin{array}
[c]{cccccc}%
\ast & 0 & 0 & 0 & 0 & 0\\
0 & \ast & \ast & 0 & 0 & 0\\
0 & \ast & \ast & 0 & 0 & 0\\
0 & 0 & 0 & \ast & \ast & 0\\
0 & 0 & 0 & \ast & \ast & 0\\
0 & 0 & 0 & 0 & 0 & \ast
\end{array}
\right]  \text{ \ \ \ }\left[  B_{2}\right]  _{B_{3}}=\left[
\begin{array}
[c]{cccccc}%
\ast & 0 & \ast & 0 & 0 & 0\\
0 & \ast & 0 & 0 & 0 & 0\\
\ast & 0 & \ast & 0 & 0 & 0\\
0 & 0 & 0 & \ast & 0 & 0\\
0 & 0 & 0 & 0 & \ast & \ast\\
0 & 0 & 0 & 0 & \ast & \ast
\end{array}
\right]
\]
\\
\[
\left[  B_{3}\right]  _{B_{4}}=\left[
\begin{array}
[c]{cccccc}%
\ast & 0 & 0 & 0 & \ast & 0\\
0 & \ast & 0 & \ast & 0 & 0\\
0 & 0 & \ast & 0 & 0 & \ast\\
0 & \ast & 0 & \ast & 0 & 0\\
\ast & 0 & 0 & 0 & \ast & 0\\
0 & 0 & \ast & 0 & 0 & \ast
\end{array}
\right].
\]

\begin{figure}[ht] 
\begin{center}
\includegraphics[scale=0.5]{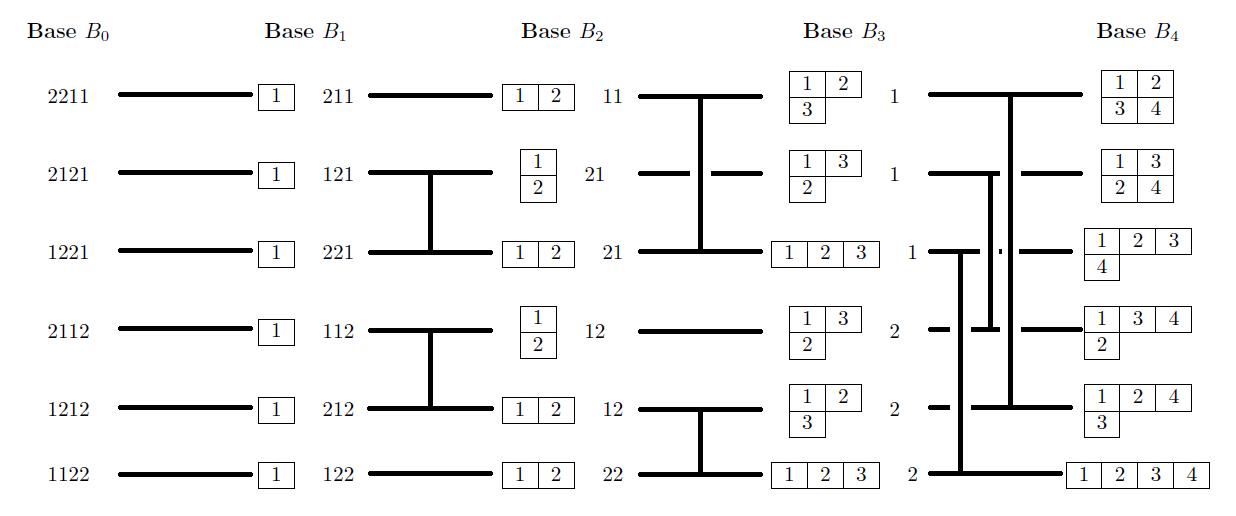} 
\caption{An illustration of the sparsity of the matrices $[B_{i-1}]_{B_{i}}$. A label of an element $b \in B_{i-1}$ is  connected with a label of an element $b' \in B_{i}$ if and only if they are related.  }
\label{Sparsity}
\end{center}
\end{figure}

\section{Connection with the Robinson-Schensted insertion algorithm}
\label{RS}

In Figure \ref{Sparsity} the vertical order of the labels of the elements of each basis $B_i$ has been carefully chosen in order to simplify the figure. In fact, the order is such that each horizontal line corresponds to a well known process: the Robinson-Schensted (RS) insertion algorithm (see  \cite{Fulton}). 

Observe that each horizontal line gives the sequence --reading from left to right-- that is obtained by applying the RS insertion algorithm to a word corresponding to an element of the basis $B_0$, which is a word in the alphabet $\{1,2\}$. The elements of this sequence are triples $(P,Q,\omega)$ where $P$ is a semistandard tableau, $Q$ is a standard tableau and $\omega$ is a word in the alphabet $\{1,2\}$. In our situation $P$ is filled with letters in $\{1,2\}$ so its height is at most $2$. It turns out that the triple $(P,Q,\omega)$ is determined by the pair $(Q,\omega)$ so $P$ can be ommited.

\begin{definition}
Let $b \in B_{i-1}$ and $b' \in B_{i}$. We say that $b$ and $b'$ are \textit{RS-related} if the label of $b'$ is obtained by applying the RS insertion step to the label of $b$. 
\end{definition}

From the definitions it is immediate the following (see Figure \ref{Sparsity} for an illustration).
\begin{thm}
If $b \in B_{i-1}$ and $b' \in B_{i}$ are RS-related then they are related.
\end{thm}

\section{Application to the computation of isotypic components}
\label{Applications}

The upper bound we obtained for the algebraic complexity of the Fourier transform can be applied to the problem of computing the isotypic projections of a given function on the Johnson graph.

For $a=0,...,s$, let $\mathcal{F}_a$ be the isotypic component of  $\mathcal{F}$ corresponding to the Young diagram $(n-a,a)$ under the action of the group $S_n$. Since these components are orthogonal and expand the space $\mathcal{F}$, given a function $f \in  \mathcal{F}$ there are uniquely determined functions $f_a \in  \mathcal{F}_a$ such that 
$$
f= \sum_{a=0}^{s} f_a
$$
For $H\subseteq \{0,...,s\}$ let $f_H$ be defined by
$$
f_H=\sum_{a \in H} f_a
$$

\begin{thm}
\label{Single isotypic component}
Assume that the matrices $[B_{i-1}]_{B_{i}}$ for $i=2,3,...,n$ have been computed. Given a column vector  $[f]_{B_0}$ with $f \in \mathcal{F}$, the column vector $[f_H]_{B_0}$ can be computed using at most $4(n-1) \binom{n}{k}$  operations.
\end{thm}
\begin{proof}
First we apply the Fourier transform to the function $f$, so that we obtain the column vector $[f]_{B_n}$ using $2( n-1) \binom{n}{k}$ operations. The basis $B_n$ is parametrized by all Young tableaux of shape $(n-a,a)$ for $a=0,...,s$. Then we substitute by $0$ the values of the entries of the vector  $[f]_{B_n}$ that correspond to Young tableaux of shape $(n-a,a)$ with $a$ not in $H$. The resulting column vector is $[f_H]_{B_n}$. Finally we apply the inverse Fourier transform to $[f_H]_{B_n}$ so that we obtain $[f_H]_{B_0}$ using $2(n-1) \binom{n}{k}$ more operations.
\end{proof}

\begin{thm}
\label{Weights of the decomposition}
Assume that the matrices $[B_{i-1}]_{B_{i}}$ for $i=2,3,...,n$ have been computed. Given a column vector  $[f]_{B_0}$ with $f \in \mathcal{F}$, all the weights $\|f_a\|^2$, for $a=0,...,s$,   can be computed using at most $(2n-1) \binom{n}{k}$ operations.
\end{thm}
\begin{proof}
Observe that $\|f_a\|^2 = \|[f_a]_{B_n}\|^2 $. To obtain the  column vector $[f_a]_{B_n}$, we apply the Fourier transform to the function $f$, so that we obtain the column vector $[f]_{B_n}$ using $2( n-1) \binom{n}{k}$ operations. Then we select the entries of the vector  $[f]_{B_n}$ that correspond to Young tableaux of shape $(n-a,a)$, and we compute the sum of the squares of these entries. Doing this for all the values of $a$ can be accomplished using at most $\binom{n}{k}$ operations.
\end{proof}

\section{Further directions}
\label{Further directions}

Our results shows that once the matrices $[B_i]_{B_{i-1}}$  have been computed for $i=2,3,...,n$, the Fourier transform of any vector $f$ can be computed in $O(n \binom{n}{k})$ operations. It emerges the problem of computing efficiently these matrices, that is, finding an algorithm that on input $(n,k)$ it computes all the matrices $[B_i]_{B_{i-1}}$ using $O(n^c \binom{n}{k})$ operations for some constant $c$. We conjecture that such an algorithm exists.

\end{document}